\documentclass[10pt,reqno]{amsart}

\usepackage{a4wide}
\usepackage{dsfont}            
\usepackage{amssymb,amsmath,amsfonts}
\usepackage{mathrsfs}
\usepackage{graphicx}          
\usepackage[latin10]{inputenc} 
\usepackage[all]{xy}

\newtheorem{proposition}{Proposition}[section]
  \newtheorem{theorem}[proposition]{Theorem}
  \newtheorem{corollary}[proposition]{Corollary}
  \newtheorem{lemma}[proposition]{Lemma}
\theoremstyle{definition}
  \newtheorem{definition}[proposition]{Definition}
  \newtheorem{remark}[proposition]{Remark}

\newcommand{\cst}{\ifmmode\mathrm{C}^*\else{$\mathrm{C}^*$}\fi}
\newcommand{\st}{\;\vline\;}
\newcommand{\ii}{\mathrm{i}}
\newcommand{\CC}{\mathbb{C}}
\newcommand{\RR}{\mathbb{R}}
\newcommand{\KK}{\mathbb{K}}
\newcommand{\GG}{\mathbb{G}}
\newcommand{\is}[2]{\left\langle#1\,\vline\,#2\right\rangle}
\newcommand{\tens}{\otimes}
\newcommand{\vtens}{\,\bar{\otimes}\,}
\newcommand{\id}{\mathrm{id}}
\newcommand{\comp}{\circ}
\newcommand{\eps}{\varepsilon}
\newcommand{\uu}{{\scriptscriptstyle\mathrm{u}}}
\newcommand{\I}{\mathds{1}}
\newcommand{\cH}{\mathscr{H}}
\newcommand{\cR}{\mathscr{R}}

\newcommand{\HH}{\mathbb{H}}
\newcommand{\sM}{\mathsf{M}}
\newcommand{\sN}{\mathsf{N}}
\newcommand{\sJ}{\mathsf{J}}
\newcommand{\sL}{\mathsf{L}}
\newcommand{\cJ}{\mathcal{J}}
\newcommand{\sR}{\mathsf{R}}
\newcommand{\hh}[1]{\widehat{#1}}
\newcommand{\op}{\text{\rm\tiny{op}}}

\newcommand{\flip}{\boldsymbol{\sigma}}
\newcommand{\ww}{\mathrm{W}}
\newcommand{\WW}{{\mathds{V}\!\!\text{\reflectbox{$\mathds{V}$}}}}
\newcommand{\Ww}{\mathds{W}}
\newcommand{\wW}{\text{\reflectbox{$\Ww$}}\:\!} 
\newcommand{\dd}[1]{\widetilde{#1}}
\newcommand{\IM}[1]{\overline{\mathrm{im}\,#1}}
\newcommand{\bh}{\boldsymbol{h}}
\newcommand{\bhL}{\boldsymbol{h}_{\text{\tiny\rm{L}}}}
\newcommand{\starr}{\,\underline{*}\,}
\newcommand{\staru}{\,\overline{*}\,}

\DeclareMathOperator{\C}{C}
\DeclareMathOperator{\B}{B}
\DeclareMathOperator{\vN}{vN}
\DeclareMathOperator{\Mor}{Mor}
\DeclareMathOperator{\M}{M}
\DeclareMathOperator{\Dom}{Dom}

\DeclareMathOperator{\Linf}{\mathnormal{L}^\infty\;\!\!}
\DeclareMathOperator{\Ltwo}{\mathnormal{L}^2\;\!\!}
\DeclareMathOperator{\Lone}{\mathnormal{L}^1\;\!\!}

\numberwithin{equation}{section}

\author{Pawe{\l} Kasprzak}
\address{Department of Mathematical Methods in Physics, Faculty of Physics, University of Warsaw, Poland}
\email{pawel.kasprzak@fuw.edu.pl}

\author{Fatemeh Khosravi}
\address{Department of Pure Mathematics, Ferdowsi University of Mashhad, Iran}
\email{fa.khosravi@stu-mail.um.ac.ir}

\author{Piotr M.~So{\l}tan}
\address{Department of Mathematical Methods in Physics, Faculty of Physics, University of Warsaw, Poland}
\email{piotr.soltan@fuw.edu.pl}

\title[Integrable actions and quantum subgroups]{Integrable actions and quantum subgroups}



\begin{document}

\begin{abstract}
We study homomorphisms of locally compact quantum groups from the point of view of integrability of the associated action. For a given homomorphism of quantum groups $\Pi\colon\mathbb{H}\to\mathbb{G}$ we introduce quantum groups $\mathbb{H}/\!\ker{\Pi}$ and $\overline{\mathrm{im}\,\Pi}$ corresponding to the classical quotient by kernel and closure of image. We show that if the action of $\mathbb{H}$ on $\mathbb{G}$ associated to $\Pi$ is integrable then $\mathbb{H}/\!\ker\Pi\cong\overline{\mathrm{im}\,\Pi}$ and characterize such $\Pi$. As a particular case we consider an injective continuous homomorphism $\Pi\colon{H}\to{G}$ between locally compact groups $H$ and $G$. Then $\Pi$ yields an integrable action of $H$ on $L^\infty\;\!\!(G)$ if and only if its image is closed and $\Pi$ is a homeomorphism of $H$ onto $\mathrm{im}\,\Pi$.

We also give characterizations of open quantum subgroups and of compact quantum subgroups in terms of integrability and show that a closed quantum subgroup always gives rise to an integrable action. Moreover we prove that quantum subgroups closed in the sense of Woronowicz whose associated homomorphism of quantum groups yields an integrable action are closed in the sense of Vaes.
\end{abstract}

\maketitle


\newlength{\sw}
\settowidth{\sw}{$\scriptstyle\sigma-\text{\rm{weak closure}}$}
\newlength{\nc}
\settowidth{\nc}{$\scriptstyle\text{\rm{norm closure}}$}
\newlength{\ssw}
\settowidth{\ssw}{$\scriptscriptstyle\sigma-\text{\rm{weak closure}}$}
\newlength{\snc}
\settowidth{\snc}{$\scriptscriptstyle\text{\rm{norm closure}}$}

\section{Introduction}\label{intro}

The theory of locally compact quantum groups is by now a well established branch of mathematical research. Once the details of the definition of a locally compact quantum group have been worked out in a series of papers (cf.~e.g.~\cite{KV,KVvN,mnw}) and a lot of work has been done on examples of such objects (\cite{PodSLW,e2,azb,nazb,bicrossed,low,VaesVainerman}), a number of deep results of the theory of locally compact groups have been generalized and analyzed from the point of view of the new theory of quantum groups (\cite{impl,induced,impr}). These papers pointed the way to a more thorough analysis of such basic concepts as subgroups (\cite{DKSS}) or homomorphisms (\cite{MRW}) as well as actions (\cite{impl}, \cite[Chapter 2]{VaesPhD}, see also \cite{exa}). 

In this paper we continue this line of research by analyzing the concept of an integrable action in the context of homomorphisms of quantum groups. We show that integrability of actions associated to homomorphisms is deeply connected with the notion of a closed quantum subgroup. For a homomorphism of quantum groups we introduce and study quantum groups which are non-commutative analogs of the closure of the image of the homomorphism and the quotient by the kernel of the homomorphism. Integrability of the associated action is then equivalent to a condition which can be interpreted as compactness of the kernel and closeness of the image together with a natural isomorphism of the image with quotient by kernel.

These considerations provide a way to show that the difference between the two notions of a closed quantum subgroup analyzed in \cite{DKSS} lies precisely in the integrability of the associated action. In particular a Woronowicz-closed quantum subgroup (see the end of Section \ref{cons} or \cite[Section 3]{DKSS}) with its associated action integrable is closed (in the sense of Vaes, see \cite{DKSS}). These results have an overlap with \cite[Proposition 3.12]{DeCommer} with also practically the same technique of proof.

We also provide characterization of open and compact quantum subgroups in terms of integrability of certain natural actions associated with them (but not with the homomorphism mapping the quantum subgroup into the ambient quantum group). The former is then used in other results on integrability, while the latter is of decidedly different nature and is hence placed in the Appendix. 

Our results are also of interest in the case the considered quantum groups are in fact classical locally compact quantum groups. In particular, if $H$ and $G$ are locally compact groups and $\Pi\colon{H}\to{G}$ is an injective continuous homomorphism then the integrability of the associated action of $H$ on $\Linf(G)$ is equivalent to $\Pi$ being a homeomorphism onto its closed image.

Let us give a more detailed account of the contents of the paper. In Section \ref{Prel} we collect the main definitions and preliminary results about operator algebras and locally compact quantum groups. We discuss in detail the notions of actions, homomorphisms and closed quantum subgroups of locally compact quantum groups. We also discuss the notion of the canonical implementation of an action (\cite{impl}) which is one of our crucial tools. Section \ref{SectIntegr} is devoted to the proof of an integrability criterion for ergodic actions. This is very much in the spirit of some of the results of \cite{induced}, but the type of actions considered is different.

In Section \ref{ImKer} we introduce the quantum groups which play the role of closure of the image and quotient by kernel of a homomorphism of locally compact quantum groups. We also generalize the latter construction to introduce quotient by the kernel of an action (an example of this construction already appeared in literature). The short Section \ref{sectCan} sheds light on the canonical implementation of the action associated to a homomorphism of quantum groups and recalls an integrability criterion from \cite{impl}. Then, in Section \ref{cons} all of our main results are obtained. 

The Appendix (Section \ref{appendix}) contains discussion of consequences of our results for classical groups and a theorem characterizing compact quantum subgroups of locally compact quantum groups in terms of integrability of the canonical action on the homogeneous space. We also recall several facts about convolutions of various classes of functionals on \cst-algebras and von Neumann algebras associated to a quantum group and prove a few results necessary for the above mentioned characterization of compact quantum subgroups. Some of these results are also used earlier (in Section \ref{sectCan}).

\section{Preliminaries}\label{Prel}

We will use the language and notation of operator algebras for locally compact quantum groups. The operator algebraic prerequisites are contained e.g.~in \cite{SZ}. Some notions related to \cst-algebras like morphisms or multipliers may be found e.g.~in \cite{unbo,gen}. We will also use the concept of an operator valued weight (\cite{ovw1}) in the context of invariant weights on locally compact quantum groups (\cite{VaesPhD,impl,KVvN}). The symbol $\flip$ will always denote the flip morphism between tensor product of operator algebras. Almost all tensor products will be denoted by the symbol $\tens$ with the precise meaning depending on the context. The only exception will be the tensor product of von Neumann algebras which we will denote by $\vtens$.

\subsection{Locally compact quantum groups}\label{lcqgs}

We refer to \cite{KV,KVvN,mnw} for fundamentals of the theory of locally compact quantum groups. Our conventions are those introduced in e.g.~\cite{BS,mu,mu2,DKSS}. Thus a locally compact quantum group $\GG$ is described by a von Neumann algebra $\Linf(\GG)$ with comultiplication $\Delta_\GG\colon\Linf(\GG)\to\Linf(\GG)\vtens\Linf(\GG)$. The right Haar weight on $\Linf(\GG)$ will be denoted by $\bh$ with corresponding GNS Hilbert space denoted by $\Ltwo(\GG)$ and associated GNS map $\eta$ (\cite[Appendix B]{mnw}). The modular operator and conjugation for $\bh$ will be denoted by $\nabla$ and $J$ respectively. We will also denote by $\Lone(\GG)$ the space $\Linf(\GG)_*$ of $\sigma$-weakly continuous functionals on $\Linf(\GG)$. The left Haar weight will be denoted by $\bhL$.

The symbols $\C_0(\GG)$ and $\C_0^\uu(\GG)$ will denote the reduced and universal \cst-algebra related to $\GG$. Since $\C_0(\GG)\subset\Linf(\GG)$ we will use $\Delta_\GG$ to also denote the comultiplication on $\C_0(\GG)$, so we have $\Delta_\GG\in\Mor(\C_0(\GG),\C_0(\GG)\tens\C_0(\GG))$, while the comultiplication on $\C_0^\uu(\GG)$ will be denoted by $\Delta_\GG^\uu$. The algebras $\Linf(\GG)$, $\C_0(\GG)$ and $\C_0^\uu(\GG)$ are endowed with much more natural structure. In particular there is the \emph{scaling group} $(\tau_t)_{t\in\RR}$ and \emph{unitary antipode} $R$ (\cite{mu,KV,mu2}).

The \emph{Kac-Takesaki operator} or the \text{right regular representation} of $\GG$ is the unitary operator $\ww^\GG$ which extends the map
\[
\eta(a)\tens\eta(b)\longmapsto(\eta\tens\eta)\bigl(\Delta(a)(\I\tens{b})\bigr)
\]
to the level of $\Ltwo(\GG)\tens\Ltwo(\GG)$.

It turns out that $\ww^\GG$ is a multiplicative unitary operator on $\Ltwo(\GG)\tens\Ltwo(\GG)$ and $\GG$ is the object related to this multiplicative unitary via the theory developed in \cite{BS,mu}. Thus, in particular
\begin{itemize}
\item $\C_0(\GG)=\bigl\{(\omega\tens\id)\ww^\GG\st\omega\in\B(\Ltwo(\GG))_*\bigr\}^{\text{\tiny{norm closure}}\hspace{-\snc}}$,
\item $\Linf(\GG)=\bigl\{(\omega\tens\id)\ww^\GG\st\omega\in\B(\Ltwo(\GG))_*\bigr\}^{\text{\tiny{$\sigma$-weak closure}}\hspace{-\ssw}}$,
\item for any $x\in\Linf(\GG)$ we have $\Delta_\GG(x)=\ww^\GG(x\tens\I){\ww^\GG}^*$.
\end{itemize}

The multiplicative unitary gives rise to the \emph{dual} $\hh{\GG}$ of $\GG$ which turns out to also be a locally compact quantum group. In particular there is a right Haar weight $\hh{\bh}$ on $\Linf(\hh{\GG})$, where
\[
\Linf(\hh{\GG})=\bigl\{(\id\tens\omega)\ww^\GG\st\omega\in\B(\Ltwo(\GG))_*\bigr\}^{\text{\tiny{$\sigma$-weak closure}}\hspace{-\ssw}}.
\]
The GNS Hilbert space for $\hh{\bh}$ can and will be naturally identified with $\Ltwo(\GG)$ and in fact $\ww^{\hh{\GG}}=\flip\bigl({\ww^\GG}^*\bigr)$ and $\ww^\GG\in\M(\C_0(\hh{\GG})\tens\C_0(\GG))$. Moreover, denoting by $\hh{\nabla}$ and $\hh{J}$ the modular operator and conjugation related to $\hh{\bh}$ we have
\[
\begin{aligned}
R(x)&=\hh{J}x^*\hh{J},&&x\in\Linf(\GG),\\
\tau_t(x)&=\hh{\nabla}^{\ii{t}}x\hh{\nabla}^{-\ii{t}},&&x\in\Linf(\GG),&t\in\RR.
\end{aligned}
\]

The results of \cite{univ,mu2,MRW} provide a lift of $\ww^\GG\in\M(\C_0(\hh{\GG})\tens\C_0(\GG))$ to an element $\WW^\GG\in\M(\C_0^\uu(\hh{\GG})\tens\C_0^\uu(\GG))$ such that if $\Lambda_{\hh{\GG}}\in\Mor(\C_0^\uu(\hh{\GG}),\C_0(\hh{\GG}))$ and $\Lambda_\GG\in\Mor(\C_0^\uu(\GG),\C_0(\GG))$ are reducing morphisms (\cite[Definition 35]{mu2}) then $\ww^\GG=(\Lambda_{\hh{\GG}}\tens\Lambda_\GG)\WW^\GG$. The ``half-lifted'' versions of $\WW^\GG$ are $\Ww^\GG=(\id\tens\Lambda_\GG)\WW^\GG\in\M(\C_0^\uu(\hh{\GG}))\tens\C_0(\GG))$ and $\wW^\GG=(\Lambda_{\hh{\GG}}\tens\id)\WW^\GG\in\M(\C_0(\hh{\GG})\tens\C_0^\uu(\GG))$.

Let $\GG$ be a locally compact quantum group and let $\sL\subset\Linf(\GG)$ be a von Neumann subalgebra. We say that $\sL$ is a \emph{left coideal} if $\Delta_\GG(\sL)\subset\Linf(\GG)\vtens\sL$. If, moreover, $\sL$ satisfies $\Delta_\GG(\sL)\subset\sL\vtens\sL$ then we say that $\sL$ is an \emph{invariant subalgebra} (this terminology was introduced in \cite{TT}). Finally an invariant subalgebra $\sL$ such that $R(\sL)=\sL$ and $\tau_t(\sL)=\sL$ for all $t\in\RR$ is called a \emph{Baaj-Vaes subalgebra}. The important Baaj-Vaes theorem (\cite[Proposition 10.5]{BV}) says that if $\sL\subset\Linf(\GG)$ is a Baaj-Vaes subalgebra then there exists a locally compact quantum group $\KK$ such that $\sL=\Linf(\KK)$ and $\Delta_\KK=\bigl.\Delta_\GG\bigr|_{\sL}$.

If $\sL\subset\Linf(\GG)$ is a left coideal then by \cite[Proposition 3.5]{eqhs} the relative commutant $\sL'\cap\Linf(\hh{\GG})$ is a left coideal in $\Linf(\hh{\GG})$ called the \emph{co-dual} of $\sL$ and is denoted by $\dd{\sL}$. By \cite[Theorem 3.9]{eqhs} we have $\dd{\dd{\sL}}=\sL$ for any left coideal $\sL\subset\Linf(\GG)$.

\subsection{Actions of quantum groups}\label{subsAct}

Let $\HH$ be a locally compact quantum group. A \emph{right action} of $\HH$ on a von Neumann algebra $\sN$ is an injective normal unital $*$-homomorphism $\alpha\colon\sN\to\sN\vtens\Linf(\HH)$ such that $(\alpha\tens\id)\comp\alpha=(\id\tens\Delta_\HH)\comp\alpha$. Since most of the time we will only use right actions, the term ``action'' will always refer to a right action. Similarly to a right action we define a \emph{left action} of $\HH$ on $\sN$ as an injective normal unital $*$-homomorphism $\beta\colon\sN\to\Linf(\HH)\vtens\sN$ such that $(\id\tens\beta)\comp\beta=(\Delta_\HH\tens\id)\comp\beta$.

Let $\alpha\colon\sN\to\sN\vtens\Linf(\HH)$ be a (right) action of $\HH$ on $\sN$. The \emph{crossed product} of $\sN$ by (the action of) $\HH$ is defined as the von Neumann subalgebra of $\sN\vtens\B(\Ltwo(\HH))$ generated by $\alpha(\sN)$ and $\I\vtens\Linf(\hh{\HH})$. The crossed product will be denoted by the symbol $\sN\rtimes_\alpha\HH$. 

Let us note that there is a simple passage from right to left actions (and conversely). Indeed, if $\alpha\colon\sN\to\sN\vtens\Linf(\HH)$ is a right action of $\HH$ on $\sN$ then $\beta=\flip\comp\alpha\colon\sN\to\Linf(\HH)\vtens\sN$ is a left action of $\HH^\op$, i.e.~the locally compact quantum group $\HH$ with reversed comultiplication $\Delta_{\HH^\op}=\flip\comp\Delta_\HH$. This allows simple translation of many results about right actions to left actions.

The central notion of this paper is that of an \emph{integrable action}. The concept of integrability of an action of a locally compact group was introduced in \cite[Definition 2.6]{ConnesTakesaki} for locally compact groups. The definition can easily be generalized to Kac algebras and locally compact quantum groups as was done e.g.~in \cite[Definition 2.3.4]{VaesPhD}, \cite[Definition 1.4]{impl}. Thus an action $\alpha\colon\sN\to\sN\vtens\Linf(\HH)$ is integrable if the set
\begin{equation}\label{integr+}
\bigl\{x\in\sN_+\st(\id\tens\bhL)\alpha(x)\in\sN_+\bigr\}
\end{equation}
is $\sigma$-weakly dense in $\sN_+$. Elements of \eqref{integr+} are called \emph{integrable} for $\alpha$. Similarly an element $x\in\sN$ is \emph{square integrable} for $\alpha$ if $x^*x$ is integrable for $\alpha$. Note that integrability of a right action is with respect to the left Haar weight. Similarly we say that a left action $\beta\colon\sN\to\Linf(\HH)\vtens\sN$ is integrable if the set
\[
\bigl\{x\in\sN_+\st(\bh\tens\id)\beta(x)\in\sN_+\bigr\}
\]
is $\sigma$-weakly dense in $\sN_+$.

Let us recall briefly some of the results of \cite{impl}. If $\alpha\colon\sN\to\sN\vtens\Linf(\HH)$ is an action of $\HH$ on $\sN$ then choosing a n.s.f.~weight $\theta$ on $\sN$ one can define a weight $\widetilde{\theta}$ on $\sN\rtimes_\alpha\HH$ (\cite[Definition 3.4 and following remarks]{impl}) such that the corresponding GNS Hilbert space is $\cH_\theta\tens\Ltwo(\HH)$. Let $\widetilde{J}$ be the modular conjugation related to this weight. Then the \emph{canonical implementation} of $\alpha$ is by definition the unitary
\[
U=\widetilde{J}(J_\theta\tens\hh{J}).
\]
Some of the major results of \cite{impl} are that
\begin{itemize}
\item $U\in\B(\cH_\theta)\vtens\Linf(\HH)$,
\item $(\id\tens\Delta_\HH)U=U_{12}U_{13}$,
\item for $x\in\sN$ we have $\alpha(x)=U(x\tens\I)U^*$.
\end{itemize}
The last formula says that $U$ \emph{implements} the action $\alpha$.

Finally let us introduce two more important classes of actions: an action $\alpha\colon\sN\to\sN\vtens\Linf(\HH)$ of a locally compact quantum group $\HH$ on a von Neumann algebra $\sN$ is \emph{ergodic} if the condition $\alpha(x)=x\tens\I$ implies that $x\in\CC\I$. The action $\alpha$ is \emph{free} if the set
\[
\bigl\{(\omega\tens\id)\alpha(x)\st{x}\in\sN,\,\omega\in\sN_*\bigr\}
\]
generates the von Neumann algebra $\Linf(\HH)$.

\subsection{Homomorphism of quantum groups}

Let $\HH$ and $\GG$ be locally compact quantum groups. The notion of a homomorphism from $\HH$ to $\GG$ can be described in one of three ways. More precisely there are natural bijections between the following three sets
\begin{itemize}
\item the set of \emph{bicharacters} from $\HH$ to $\GG$, i.e.~unitaries $V\in\Linf(\hh{\GG})\vtens\Linf(\HH)$ such that
\[
(\id\tens\Delta_\HH)V=V_{12}V_{13}\quad\text{and}\quad(\Delta_{\hh{\GG}}\tens\id)V=V_{23}V_{13},
\]
\item the set of actions $\alpha\colon\Linf(\GG)\to\Linf(\GG)\vtens\Linf(\HH)$ such that
\[
(\Delta_\GG\tens\id)\comp\alpha=(\id\tens\alpha)\comp\Delta_\GG
\]
(such actions are called \emph{right quantum group homomorphisms}, cf.~\cite[Section 1]{MRW}),
\item the set of those $\pi\in\Mor(\C_0^\uu(\GG),\C_0^\uu(\HH))$ which satisfy
\[
\Delta_\HH^\uu\comp\pi=(\pi\tens\pi)\comp\Delta_\GG^\uu
\]
(such morphisms are referred to as \emph{Hopf $*$-homomorphisms}).
\end{itemize}
Each element of either of the three sets described above represents a \emph{homomorphism} from $\HH$ to $\GG$. Thus from now on we will write $\Pi\colon\HH\to\GG$ to denote a homomorphism from $\HH$ to $\GG$ and freely use its three ``incarnations'', namely a bicharacter $V\in\Linf(\hh{\GG})\vtens\Linf(\HH)$, a right quantum group homomorphism $\alpha\colon\Linf(\GG)\to\Linf(\GG)\vtens\Linf(\HH)$ and a Hopf $*$-homomorphism $\pi\in\Mor(\C_0^\uu(\GG),\C_0^\uu(\HH))$. The relationships between $V$, $\alpha$ and $\pi$ representing the same $\Pi$ are
\[
\begin{split}
V&=\bigl(\id\tens(\Lambda_\HH\comp\pi)\bigr)\wW^\GG,\\
\alpha\comp\Lambda_\GG&=(\Lambda_\GG\tens\Lambda_\HH)\comp(\id\tens\pi)\comp\Delta_\GG^\uu
\end{split}
\]
and
\begin{equation}\label{Vimpl}
\alpha(x)=V(x\tens\I)V^*,\qquad{x}\in\Linf(\GG).
\end{equation}

\subsection{Closed quantum subgroups and open quantum subgroups}\label{qsub}

Let $\GG$ and $\HH$ be locally compact quantum groups. Following \cite[Definition 2.5]{impr} (cf.~\cite{DKSS} and the discussion therein) we say that a homomorphism of quantum groups $\Pi\colon\HH\to\GG$ identifies $\HH$ with a \emph{closed quantum subgroup} of $\GG$ if there exists an injective normal $*$-homomorphism $\gamma\colon\Linf(\hh{\HH})\to\Linf(\hh{\GG})$ such that the corresponding bicharacter $V$ is given by $(\gamma\tens\id)\ww^\HH$. Equivalently one can only demand that $\gamma$ satisfy $\Delta_{\hh{\GG}}\comp\gamma=(\gamma\tens\gamma)\comp\Delta_{\hh{\HH}}$ (\cite[Theorem 3.3]{DKSS}).

It follows from \cite[Proposition 5.45]{KV} that if $\HH$ is a closed subgroup of $\GG$ then the image of $\Linf(\hh{\HH})$ under the injection $\gamma$ is a Baaj-Vaes subalgebra of $\Linf(\hh{\GG})$. In particular this shows that there is a bijection between closed quantum subgroups of $\GG$ and Baaj-Vaes subalgebras of $\Linf(\hh{\GG})$.

Now let $\GG$ and $\HH$ be locally compact quantum groups. According to \cite[Definition 2.2]{open} we say that $\HH$ is an \emph{open quantum subgroup} of $\GG$ if there is a surjective normal $*$-homomorphism $\Theta\colon\Linf(\GG)\to\Linf(\HH)$ such that $\Delta_\HH\comp\Theta=(\Theta\tens\Theta)\comp\Delta_\GG$. In this case the algebra $\Linf(\HH)$ can be identified with a corner of $\Linf(\GG)$ defined by the central support of $\Theta$ which we denote by the symbol $\I_\HH$. Let us denote the map $\Linf(\HH)\to\I_\HH\Linf(\GG)\subset\Linf(\GG)$ by $\iota$. The composition $\iota\comp\Theta$ is the conditional expectation
\[
\Linf(\GG)\ni{x}\longmapsto\I_\HH{x}\in\Linf(\GG).
\]

By \cite[Theorem 3.6]{open} an open quantum subgroup of $\GG$ is a closed quantum subgroup of $\GG$. In fact the corresponding bicharacter $V$ is $(\id\tens\Theta)\ww^\GG$ (cf.~\cite[Remark 3.7]{open}) and the image of $\Linf(\hh{\HH})$ in $\Linf(\hh{\GG})$ under the corresponding injection is the $\sigma$-weak closure of
\[
\bigl\{(\id\tens\omega)(\id\tens\Theta)\ww^\GG\st\omega\in\B(\Ltwo(\HH))_*\bigr\}.
\]

\section{Integrable ergodic actions and open quantum subgroups}\label{SectIntegr}

In this section we will prove a result on integrability of an action (Proposition \ref{PropIntegr}) which will enable us to give a characterization of open subgroups in terms of existence of certain integrable elements (Theorem \ref{openThm}). This result will in turn become useful in Section \ref{cons}. 

The following lemma is a generalization of \cite[Lemma 6.4]{KV}, cf.~also \cite[Theorem 2.1]{MRW}. The technique of proof is virtually identical to that of proof of \cite[Lemma 6.4]{KV}.

\begin{lemma}\label{leqlem}
Let $\alpha\colon\sN\to\sN\vtens\Linf(\GG)$ be an ergodic action of a locally compact quantum group $\GG$ on a von Neumann algebra $\sN$. If $P\in\sN$ is a projection such that $\alpha(P)\leq{P}\tens\I$ then $P=0$ or $P=\I$.
\end{lemma}

\begin{proof}
Since $\alpha$ is an injective map, $\sN$ can be embedded into $\sN\vtens\Linf(\GG)$ via this mapping. Viewing $\alpha$ as this embedding we shall denote it by $\iota$. Using this notation we have
\[
(\id\tens\Delta_\GG)\comp\iota=(\id\tens\Delta_\GG)\comp\alpha=(\alpha\tens\id)\comp\alpha=(\iota\tens\id)\tens\alpha
\]
which shows that under the identification of $\sN$ with $\iota(\sN)$ the action $\alpha$ is given by $\id\tens\Delta_\GG$. In particular, as $\alpha$ is ergodic, $\id\tens\Delta_\GG$ is ergodic on $\iota(\sN)$. Therefore it is enough to prove the lemma for $\bigl.(\id\tens\Delta_\GG)\bigr|_{\iota(\sN)}$.

Denote $\iota(P)$ by $Q$. The condition that 
\begin{equation}\label{Q1}
(\id\tens\Delta_\GG)(Q)\leq{Q}\tens\I 
\end{equation}
means $(\id\tens\Delta_\GG)(Q)Q_{12}=(\id\tens\Delta_\GG)(Q)$
or
\[
\ww^\GG_{23}Q_{12}{\ww^\GG_{23}}^*Q_{12}=\ww^\GG_{23}Q_{12}{\ww^\GG_{23}}^*,
\]
which can be rewritten as 
\begin{equation}\label{WQ}
\ww^\GG_{23}Q_{12}=Q_{12}\ww^\GG_{23}Q_{12}
\end{equation}

Let $\hh{\nabla}$ and $\hh{J}$ be the modular operator and modular conjugation for $\hh{\bh}$ --- the right Haar weight of $\hh{\GG}$. By \cite[Lemma 2.1 \& Theorem 3.5]{mnw} applied to $\hh{\GG}$, for all $\xi\in\Dom\bigl(\hh{\nabla}^{-\frac{1}{2}}\bigr)$ and all $\zeta\in\Dom\bigl(\hh{\nabla}^{\frac{1}{2}}\bigr)$ we have
\begin{equation}\label{nabla}
\bigl((\id\tens\omega_{\xi,\zeta})\ww^\GG\bigr)^*=(\id\tens\omega_{\hh{J}\hh{\nabla}^{-\frac{1}{2}}\xi,\hh{J}\hh{\nabla}^{\frac{1}{2}}\zeta})\ww^\GG.
\end{equation}
Therefore for such $\xi$ and $\zeta$ we can apply $\id\tens\id\tens\omega_{\hh{J}\hh{\nabla}^{-\frac{1}{2}}\xi,\hh{J}\hh{\nabla}^{\frac{1}{2}}\zeta}$ to both sides of \eqref{WQ} to get
\[
\bigl(\I\tens(\id\tens\omega_{\hh{J}\hh{\nabla}^{-\frac{1}{2}}\xi,\hh{J}\hh{\nabla}^{\frac{1}{2}}\zeta})(\ww^\GG)\bigr)Q
=Q\bigl(\I\tens(\id\tens\omega_{\hh{J}\hh{\nabla}^{-\frac{1}{2}}\xi,\hh{J}\hh{\nabla}^{\frac{1}{2}}\zeta})(\ww^\GG)\bigr)Q
\]
which by \eqref{nabla} is equivalent to
\[
\bigl(\I\tens(\id\tens\omega_{\xi,\zeta})(\ww^\GG)\bigr)^*Q=Q\bigl(\I\tens(\id\tens\omega_{\xi,\zeta})(\ww^\GG)\bigr)^*Q
\]
or
\[
Q\bigl(\I\tens(\id\tens\omega_{\xi,\zeta})(\ww^\GG)\bigr)=Q\bigl(\I\tens(\id\tens\omega_{\xi,\zeta})(\ww^\GG)\bigr)Q.
\]
It follows that $Q_{12}\ww^\GG_{23}=Q_{12}\ww^\GG_{23}Q_{12}$ which means that $Q_{12}=Q_{12}(\id\tens\Delta_{\GG})(Q)$, i.e.
\begin{equation}\label{Q2}
Q\tens\I\leq(\id\tens\Delta_\GG)(Q). 
\end{equation}

Combining \eqref{Q1} and \eqref{Q2} we find that $(\id\tens\Delta_\GG)(Q)=Q\tens\I$, which by ergodicity of $(\id\tens\Delta_\GG)$ on $\sN$ implies that $Q=0$ or $Q=\I$.
\end{proof}

The next proposition is very similar to \cite[Proposition 6.2]{induced}. The difference is that the action considered in \cite{induced} is the action of a closed quantum subgroup on the ambient quantum group. In particular it is not ergodic, but has a particularly simple form.

\begin{proposition}\label{PropIntegr}
Let $\alpha\colon\sN\to\sN\vtens\Linf(\GG)$ be an ergodic action of a locally compact quantum group $\GG$ on a von Neumann algebra $\sN_+$. Then $\alpha$ is integrable if and only if there exists a non-zero $x\in\sN$ integrable for $\alpha$.
\end{proposition}

\begin{proof}
Clearly if $\alpha$ is integrable then there are plenty of integrable elements. Assume now that there exists a non-zero integrable element.

Let $\cJ=\bigl\{x\in\sN\st{x^*x}\text{ is integrable for }\alpha\bigr\}$. Then $\cJ$ is a non-zero left ideal in $\sN$ and so its $\sigma$-weak closure $\sJ$ is a principal left ideal, i.e.~there exists a non-zero projection $P\in\sN$ such that $\sJ=\sN{P}$. 

Let $x\in\sN_+$ be integrable for $\alpha$. Let $\delta$ be the modular element for $\GG$ and let $\xi\in\Dom\bigl(\delta^{\frac{1}{2}}\bigr)$. Consider $\omega\in\Lone(\GG)_+$ given by $\omega(y)=\is{\xi}{y\xi}$ for all $y\in\Linf(\GG)$. Using right $\delta$-invariance of $\bhL$ we get
\[
(\id\tens\bhL)\Bigl(\alpha\bigl((\id\tens\omega)\alpha(x)\bigr)\Bigr)=\bigl\|\delta^{\frac{1}{2}}\xi\bigr\|^2(\id\tens\bhL)(\alpha(x)).
\]
i.e.~$(\id\tens\omega)\alpha(x)$ is integrable. Now for any $y\in\sN$ and $\omega\in\Lone(\GG)$ as above the Kadison inequality gives
\[
\bigl((\id\tens\omega)\alpha(y)\bigr)^*\bigl((\id\tens\omega)\alpha(y)\bigr)\leq\|\omega\|(\id\tens\omega)\alpha(y*y)
\]
and thus it follows that for any $y\in\cJ$ we have $(\id\tens\omega)\alpha(y)\in\cJ$. By continuity for any $y\in\sJ$ we have $(\id\tens\omega)\alpha(y)\in\sJ$. In particular $(\id\tens\omega)\alpha(P)\in\sJ$ and consequently $(\id\tens\omega)\alpha(P)=(\id\tens\omega)\alpha(P)P$. As this is true for separating set of functionals $\omega$, we conclude that $\alpha(P)(P\tens\I)=\alpha(P)$, i.e.~$\alpha(P)\leq{P\tens\I}$. By lemma \ref{leqlem} $P=\I$, so $\cJ$ is $\sigma$-weakly dense in $\sN$.
\end{proof}

Proposition \ref{PropIntegr} has its obvious analog for left actions (cf.~Section \ref{subsAct}).

Now let $\GG$ and $\HH$ be locally compact quantum groups and assume that $\HH$ is an open subgroup of $\GG$. As mentioned in Section \ref{qsub} $\HH$ is then also a closed quantum subgroup and the image of $\Linf(\hh{\HH})$ in $\Linf(\hh{\GG})$ under the corresponding embedding is the $\sigma$-weak closure of
\begin{equation}\label{linfhh}
\bigl\{(\id\tens\omega)(\id\tens\Theta)\ww^\GG\st\omega\in\B(\Ltwo(\HH))_*\bigr\},
\end{equation}
where $\Theta\colon\Linf(\GG)\to\Linf(\HH)$ is the surjection corresponding to the embedding $\HH\subset\GG$. 

Let $\iota\colon\Linf(\HH)\hookrightarrow\Linf(\GG)$ be the (non-unital) embedding as in Section \ref{qsub}. Now take $\omega$ in \eqref{linfhh} to be of the form $\vartheta\comp\iota$, where $\vartheta\in\B(\Ltwo(\GG))_*$. Then
\[
(\id\tens\omega)\bigl((\id\tens\Theta)(\ww^\GG)\bigr)=(\id\tens\vartheta)\bigl((\I\tens\I_\HH)\ww^\GG\bigr)=(\id\tens\vartheta\I_\HH)\ww^\GG,
\]
where $\I_\HH$ is the central support of $\Theta$ and $\vartheta\I_\HH$ denotes the functional $x\mapsto\vartheta(\I_\HH{x})$. 

Now assume that $\vartheta$ is a \emph{$\Ltwo$-bounded} functional on $\Linf(\GG)$. This means that there exists a vector $\xi_\vartheta\in\Ltwo(\GG)$ such that
\[
\vartheta(x)=\is{\xi_\vartheta}{\eta(x)}
\]
for all $x\in\Dom(\eta)$. It is easy to see that for any $c\in\Linf(\GG)$ the functional $\vartheta{c}$ is still $\Ltwo$-bounded:
\[
(\vartheta{c})(x)=\vartheta(cx)=\is{\xi_\vartheta}{\eta(cx)}=\is{\xi_\vartheta}{c\eta(x)}=\is{c^*\xi_\vartheta}{\eta(x)}
\]
for all $x\in\Dom(\eta)$. It is known (\cite[Section 6]{mnw}) that the set of $\Ltwo$-bounded functionals is norm dense in $\Lone(\GG)$.

\begin{theorem}\label{openThm}
Let $\HH$ be a closed quantum subgroup of $\GG$ via $\gamma\colon\Linf(\hh{\HH})\hookrightarrow\Linf(\hh{\GG})$. Then $\HH$ is open in $\GG$ if and only if there exists a non-zero element $x\in\Linf(\hh{\HH})$ such that $\gamma(x)$ is square-integrable with respect to the right Haar measure of $\hh{\GG}$.
\end{theorem}

\begin{proof}
We can identify $\Linf(\hh{\HH})$ with its image under $\gamma$, so $\Linf(\hh{\HH})\subset\Linf(\hh{\GG})$. Assume that $x\in\Linf(\hh{\HH})$ is non-zero and square integrable for $\hh{\bh}$. The map $\Delta_{\hh{\GG}}$ restricted to $\Linf(\hh{\HH})$ can be viewed as an ergodic left action $\alpha$ of $\hh{\GG}$ on $\hh{\HH}$:
\[
\beta=\bigl.\Delta_{\hh{\GG}}\bigr|_{\Linf(\hh{\HH})}\colon\Linf(\hh{\HH})\longrightarrow\Linf(\hh{\GG})\vtens\Linf(\hh{\HH})
\]
and the element $x^*x$ is non-zero and integrable for $\beta$. By a left analog of Proposition \ref{PropIntegr} $\beta$ is integrable, and by right invariance, $\hh{\bh}$ restricted to $\Linf(\hh{\HH})$ is semifinite. Now \cite[Theorem 7.5]{open} guarantees that $\HH$ is open.

Assume now that $\HH$ is open. From the discussion preceding the statement of the theorem we know that the image of $\Linf(\hh{\HH})$ in $\Linf(\hh{\GG})$ is the closure of the set
\[
\bigl\{(\id\tens\vartheta\I_\HH)\ww^\GG\st\vartheta\in\Lone(\GG)\bigr\}.
\]
Since the set of $\Ltwo$-bounded functionals is dense in $\Lone(\GG)$, there must exist an $\Ltwo$-bounded $\vartheta$ such that $x=(\id\tens\vartheta\I_\HH)\ww^\GG$ is non-zero. Moreover, we also know that $\vartheta\I_\HH$ is $\Ltwo$-bounded, so $x^*\in\Dom(\hh{\eta})$ by \cite[Theorem 6.4(4)]{mnw}.
\end{proof}

\begin{corollary}\label{openCor}
Let $\GG$ and $\HH$ be locally compact quantum groups and let $\HH$ be a closed quantum subgroup of $\GG$ with corresponding injection $\gamma\colon\Linf(\hh{\HH})\hookrightarrow\Linf(\hh{\GG})$. Let $\hh{\bh}$ be the right Haar weight on $\Linf(\hh{\GG})$. Then $\HH$ is open in $\GG$ if and only if the right Haar measure of $\hh{\HH}$ is proportional to $\hh{\bh}\comp\gamma$.
\end{corollary}

\begin{proof}
This has practically been proved in Theorem \ref{openThm}. To make things more precise, we note that the condition that the Haar weight on $\Linf(\hh{\HH})$ is proportional to $\hh{\bh}\comp\gamma$ is equivalent to $\hh{\bh}\comp\gamma$ being semifinite on $\Linf(\hh{\HH})$. Thus if $\hh{\bh}\comp\gamma$ is semifinite then $\HH$ is open in $\GG$ by Theorem \ref{openThm}. Conversely if $\HH$ is open in $\GG$ then Theorem \ref{openThm} says that there is a non-zero element of $\gamma\bigl(\Linf(\hh{\HH})\bigr)$ integrable for $\hh{\bh}$. By right invariance this element is integrable for the ergodic action of $\hh{\GG}$ on $\gamma\bigl(\Linf(\hh{\HH})\bigr)$ given by $\Delta_{\hh{\HH}}$. Thus by the left analog of Proposition \ref{PropIntegr} this action is integrable, which again by right invariance means that $\hh{\bh}\comp\gamma$ is semifinite.
\end{proof}

In terms of integrability of actions the characterization of open quantum subgroups is given by the following corollary:

\begin{corollary}\label{openInt}
Let $\GG$ and $\HH$ be locally compact quantum groups and let $\HH$ be a closed subgroup of $\GG$ with corresponding injection $\gamma\colon\Linf(\hh{\HH})\hookrightarrow\Linf(\hh{\GG})$. Then $\HH$ is open in $\GG$ if and only if the left action of $\hh{\GG}$ on $\gamma\bigl(\Linf(\hh{\HH})\bigr)$ given by $\Delta_{\hh{\GG}}$ is integrable.
\end{corollary}

\section{Image and kernel of a homomorphism}\label{ImKer}

Let $\GG$ and $\HH$ be locally compact quantum groups throughout this section we will be focused on a homomorphism $\Pi\colon\HH\to\GG$ with corresponding bicharacter $V\in\Linf(\hh{\GG})\vtens\Linf(\HH)$ and a right quantum group homomorphism $\alpha\colon\Linf(\GG)\to\Linf(\GG)\vtens\Linf(\HH)$. We will introduce a closed quantum subgroup of $\GG$ and a ``quotient'' quantum group of $\HH$ which correspond to the closure of the image and quotient by the kernel of $\Pi$. 

\subsection{Closure of image of $\Pi$}

The next proposition is a slightly stronger version of \cite[Lemma 1.2]{open}.

\begin{proposition}\label{BV2}
Let $\GG$ and $\HH$ be locally compact quantum groups and let $\Pi\colon\HH\to\GG$ be a homomorphism of quantum groups with corresponding bicharacter $V\in\Linf(\hh{\GG})\vtens\Linf(\HH)$. Let
\[
\sL={\bigl\{(\id\tens\zeta)V\st\zeta\in\Lone(\HH)\bigr\}}^{\sigma-\text{\rm{weak closure}}\hspace{-\sw}}.
\]
Then $\sL$ is a Baaj-Vaes subalgebra of $\Linf(\hh{\GG})$.
\end{proposition}

\begin{proof}
First we note that $\sL$ is a von Neumann algebra. Indeed this follows directly from \cite[Theorem 1.6(1)]{mu}.  To check that $\sL$ is an invariant subalgebra it is enough to see that for any $\zeta\in\Lone(\HH)$ and any $\phi\in\Lone(\hh{\GG})$ we have
\[
(\id\tens\phi)\bigl(\Delta_{\hh{\GG}}\bigl((\id\tens\zeta)V\bigr)\bigr),
(\phi\tens\id)\bigl(\Delta_{\hh{\GG}}\bigl((\id\tens\zeta)V\bigr)\bigr)\in\sL.
\]
This we infer from the fact that
\[
(\id\tens\phi)\bigl(\Delta_{\hh{\GG}}\bigl((\id\tens\zeta)V\bigr)\bigr)=(\id\tens\phi\tens\zeta)(V_{23}V_{13})=
(\id\tens\widetilde{\phi})V,
\]
where $\widetilde{\phi}(y)=(\phi\tens\zeta)\bigl(V(\I\tens{y})\bigr)$ for all $y\in\Linf(\HH)$. Similarly
\[
(\phi\tens\id)\bigl(\Delta_{\hh{\GG}}\bigl((\id\tens\zeta)V\bigr)\bigr)=(\phi\tens\id\tens\zeta)(V_{23}V_{13})=
(\id\tens\breve{\phi})V,
\]
with $\breve{\phi}(y)=(\phi\tens\zeta)\bigl((\I\tens{y})V\bigr)$ for all $y\in\Linf(\HH)$.

What is left is to establish invariance of $\sL$ under the unitary antipode $R^\HH$ and the scaling group $\tau^\HH$ of $\HH$. This follows from the fact that for each $\zeta\in\Lone(\HH)$ and $t\in\RR$ we have
\[
\begin{aligned}
R^{\hh{\GG}}\bigl((\id\tens\zeta)V\bigr)
&=\bigl(\id\tens[\zeta\comp{R^{\HH}}]\bigr)(R^{\hh{\GG}}\tens{R^\HH})V=\bigl(\id\tens[\zeta\comp{R^{\HH}}]\bigr)V\in\sL
\\
\tau_t^{\hh{\GG}}\bigl((\id\tens\zeta)V\bigr)
&=\bigl(\id\tens[\zeta\comp{\tau_{-t}^{\HH}}]\bigr)(\tau_t^{\hh{\GG}}\tens{\tau_t^\HH})V=
\bigl(\id\tens[\zeta\comp{\tau_{-t}^{\HH}}]\bigr)V\in\sL
\end{aligned}
\]
by \cite[Proposition 3.10]{MRW} (cf.~\cite[Remark 41]{mu2}).
\end{proof}

As mentioned in Section \ref{qsub}, a Baaj-Vaes subalgebra $\sL$ of $\Linf(\hh{\GG})$ defines a quantum subgroup $\KK$ of $\GG$ by setting $\hh{\KK}$ to be the quantum group associated with $\sL$. Thus, using Proposition \ref{BV2} we can introduce the following definition:

\begin{definition}\label{DefImPi}
Let $\GG$ and $\HH$ be locally compact quantum groups and let $\Pi\colon\HH\to\GG$ be a homomorphism of quantum groups. We define the closed quantum subgroup $\IM{\Pi}$ of $\GG$ as the quantum subgroup related to the Baaj-Vaes subalgebra $\sL$ of $\Linf(\hh{\GG})$ described in Proposition \ref{BV2}. In particular 
\[
\Linf\bigl(\hh{\IM{\Pi}}\:\!\bigr)={\bigl\{(\id\tens\zeta)V\st\zeta\in\Lone(\HH)\bigr\}}^{\sigma-\text{\rm{weak closure}}\hspace{-\sw}}.
\]
\end{definition}

\subsection{Quotient by kernel of $\Pi$}

In a manner completely analogous to the one used in the proof of Proposition \ref{BV2} we obtain also the following result:

\begin{proposition}\label{BV1}
Let $\GG$ and $\HH$ be locally compact quantum groups and let $\Pi\colon\HH\to\GG$ be a homomorphism of quantum groups with corresponding bicharacter $V\in\Linf(\hh{\GG})\vtens\Linf(\HH)$. Let
\[
\sR={\bigl\{(\phi\tens\id)V\st\phi}\in\Lone(\hh{\GG})\bigr\}^{\sigma-\text{\rm{weak closure}}\hspace{-\sw}}.
\]
Then $\sR$ is a Baaj-Vaes subalgebra of $\Linf(\HH)$.
\end{proposition}

\begin{definition}\label{DefKerPi}
Let $\GG$ and $\HH$ be locally compact quantum groups. Given a homomorphism $\Pi\colon\HH\to\GG$ described by a bicharacter $V\in\Linf(\hh{\GG})\vtens\Linf(\HH)$ we define a locally compact quantum group $\HH/\!\ker\Pi$ by setting
\[
\Linf(\HH/\!\ker\Pi)={\bigl\{(\phi\tens\id)V\st\phi}\in\Lone(\hh{\GG})\bigr\}^{\sigma-\text{\rm{weak closure}}\hspace{-\sw}}
\]
and $\Delta_{\HH/\!\ker\Pi}=\bigl.\Delta_\HH\bigr|_{\Linf(\HH/\!\ker\Pi)}$.
\end{definition}

Given $\GG$, $\HH$ and $\Pi$ as in Definition \ref{DefKerPi} it may happen that the subalgebra $\Linf(\HH/\!\ker\Pi)$ is all of $\Linf(\HH)$. In this case we say that $\HH/\!\ker\Pi=\HH$. Using this language we can rewrite implication (1) $\Rightarrow$ (3) of \cite[Theorem 3.4]{DKSS} in the following way:

\begin{corollary}
Let $\GG$ and $\HH$ be locally compact quantum groups and let $\Pi\colon\HH\to\GG$ be a homomorphism of quantum groups. If $\Pi$ identifies $\HH$ with a closed subgroup of $\GG$ then $\HH/\!\ker\Pi=\HH$.
\end{corollary}

\begin{remark}\label{remImKer}
If $G$ and $H$ are locally compact groups and $\Pi\colon{H}\to{G}$ is a continuous homomorphism then it is not difficult to see that Definitions \ref{DefImPi} and \ref{DefKerPi} yield quantum groups corresponding to the closure of the image of $\Pi$ and kernel of $\Pi$ respectively.
\end{remark}

\subsection{Quotient by kernel of a quantum group action}

Let $\GG$ and $\HH$ be locally compact quantum groups. A homomorphism of quantum groups $\Pi\colon\HH\to\GG$ defines the quantum group $\HH/\!\ker\Pi$ as in Definition \ref{DefKerPi}. The next theorem shows that this quantum group can be defined directly from the action $\alpha\colon\Linf(\GG)\to\Linf(\GG)\vtens\Linf(\HH)$ associated to $\Pi$.

\begin{theorem}\label{ThmPi}
Let $\HH$ and $\GG$ be locally compact quantum groups and let $\Pi\colon\HH\to\GG$ be a homomorphism of quantum groups with corresponding bicharacter $V\in\Linf(\hh{\GG})\vtens\Linf(\HH)$ and right quantum group homomorphism $\alpha\colon\Linf(\GG)\to\Linf(\GG)\vtens\Linf(\HH)$. Then
\begin{equation}\label{MN}
{\bigl\{(\omega\tens\id)\alpha(x)\st{x}}\in\Linf(\GG),\:\omega\in\Lone(\GG)\bigr\}''=\Linf(\HH/\!\ker\Pi).
\end{equation}
\end{theorem}

\begin{proof}
Let us denote by $\sM$ the left hand side of \eqref{MN} and let $\sR$ be a shorthand for $\Linf(\HH/\!\ker\Pi)$. Clearly $\sM$ is a von Neumann algebra. The next thing to see is that $\sM\subset\Linf(\HH)$ is a left coideal. To check this take $x\in\Linf(\GG)$, $\omega\in\Lone(\GG)$ and $\zeta\in\Lone(\HH)$. We have
\[
(\zeta\tens\id)\Delta_\HH\bigl((\omega\tens\id)\alpha(x)\bigr)=\bigl(\bigl[(\omega\tens\zeta)\comp\alpha\bigr]\tens\id\bigr)\alpha(x)\in\sM,
\]
so $\Delta_\HH\bigl(\sM\bigr)\subset\Linf(\HH)\vtens\sM$. 

Once we know that $\sM$ is a left coideal in $\Linf(\GG)$, we can consider its co-dual $\dd{\sM}$ (cf.~the end of Section \ref{lcqgs} and \cite[Section 3]{eqhs}).

The co-dual of $\sM$ consists of those $y\in\Linf(\hh{\HH})$ which satisfy $uy=yu$ for all $u\in\sM$. Thus $y\in\dd{\sM}$ if and only if
\[
V(x\tens\I)V^*(\I\tens{y})=(\I\tens{y})V(x\tens\I)V^*,\qquad{x}\in\Linf(\GG)
\]
or, equivalently,
\[
V^*(\I\tens{y})V(x\tens\I)=(x\tens\I)V^*(\I\tens{y})V,\qquad{x}\in\Linf(\GG).
\]
This means that $V^*(\I\tens{y})V\in\Linf(\GG)'\vtens\B(\Ltwo(\HH))$, but at the same time $V^*(\I\tens{y})V\in\Linf(\hh{\GG})\vtens\B(\Ltwo(\HH))$. Since $\Linf(\hh{\GG})\cap\Linf(\GG)'=\CC\I$ (this is \cite[Proposition 4.7(3)]{mnw} applied to $\hh{\GG}$), we find that $V^*(\I\tens{y})V=\I\tens{z}$ for some $z\in\B(\Ltwo(\HH))$. This is equivalent to
\[
(\I\tens{y})V=V(\I\tens{z})
\]
and slicing the left leg of this equality yields
\[
yu=uz,\qquad{u}\in\sR.
\]
Since $\I\in\sR$, we conclude that $y=z$, so for $y$ to belong to $\dd{\sM}$ is equivalent to $y\in\Linf(\hh{\HH})\cap\sR'=\dd{\sR}$. In other words $\dd{\sM}=\dd{\sR}$ and so $\sM=\dd{\dd{\sM}}=\dd{\dd{\sR}}=\sR$.
\end{proof}

In analogy with the way $\Linf(\HH/\!\ker\Pi)$ is defined in terms of the action of $\HH$ on $\GG$ given by the homomorphism $\Pi$ in Theorem \ref{ThmPi} we can define a similar object for an arbitrary action $\alpha$ of $\HH$ on a von Neumann algebra $\sM$.

\begin{proposition}
Let $\sM$ be a von Neumann algebra and let $\HH$ be a locally compact quantum group. Let $\alpha\colon\sM\to\sM\vtens\Linf(\HH)$ be an action of $\HH$ on $\sM$. Let $\sN_0={\bigl\{(\omega\tens\id)\alpha(x)\st}\omega\in\sM_*,\:x\in\sM\bigr\}''$. Then $\sN_0$ is an invariant subalgebra of $\Linf(\HH)$.
\end{proposition}

\begin{proof}
The proof is completely analogous to the first stage of the proof of Theorem \ref{ThmPi}.
\end{proof}

\begin{definition}\label{DefKerAlpha}
Let $\sM$ be a von Neumann algebra and let $\HH$ be a locally compact quantum group. Let $\alpha\colon\sM\to\sM\vtens\Linf(\HH)$ be an action of $\HH$ on $\sM$. Let $\sN$ be the smallest Baaj-Vaes subalgebra of $\Linf(\HH)$ containing the set
\[
\bigl\{(\omega\tens\id)\alpha(x)\st\omega\in\sM_*,\:x\in\sM\bigr\}.
\]
We define the quantum group $\HH/\!\ker\alpha$ as the locally compact quantum group such that
\[
\Linf(\HH/\!\ker\alpha)=\sN
\]
and $\Delta_{\HH/\!\ker\alpha}=\bigl.\Delta_\HH\bigr|_{\Linf(\HH/\!\ker\alpha)}$.
\end{definition}

\begin{remark}\label{remPiAlpha}
\noindent
\begin{enumerate}
\item 
Theorem \ref{ThmPi} makes it clear that if $\HH$ and $\GG$ are locally compact quantum groups and $\Pi\colon\HH\to\GG$ is a homomorphism of quantum groups with associated right quantum group morphism (action) $\alpha\colon\Linf(\GG)\to\Linf(\GG)\vtens\Linf(\HH)$ then $\HH/\!\ker\alpha$ coincides with $\HH/\!\ker\Pi$.
\item\label{remPiAlpha2} The quantum group of inner automorphisms of a locally compact quantum group defined in \cite[Definition 3.3]{centers} is a special case of the group of the form $\HH/\!\ker\alpha$ for a certain action $\alpha$.
\item If $H$ is a locally compact group acting on a von Neumann algebra $\sN$ with corresponding ``quantum action'' $\alpha\colon\sN\to\sN\vtens\Linf(H)$ then the quantum group $H/\!\ker\alpha$ defined in Definition \ref{DefKerAlpha} corresponds to the quotient of $H$ by the subgroup of all elements which act trivially on $\sN$ (cf.~Remark \ref{remImKer}).
\end{enumerate}
\end{remark}

The next result shows that if an action $\alpha\colon\sN\to\sN\vtens\Linf(\HH)$ of a locally compact quantum group $\HH$ on a von Neumann algebra $\sN$ is integrable, the kernel of $\alpha$ is in some sense small, or rather compact (cf.~\cite[Theorem 7.2]{open}).

\begin{proposition}\label{PropNH}
Let $\sN$ be a von Neumann algebra and let $\HH$ be a locally compact quantum group. Let $\alpha\colon\sN\to\sN\vtens\Linf(\HH)$ be an action of $\HH$ on $\sN$. Assume that $\alpha$ is integrable. Then $\hh{\HH/\!\ker\alpha}$ is an open quantum subgroup of $\hh{\HH}$.
\end{proposition}

\begin{proof}
Since $\alpha$ is integrable, there exists $x\in\sN_+$ such that $(\id\tens\bhL)\alpha(x)\in\sN_+$ and there is a normal state $\omega$ on $\sN$ such that
$\omega\bigl((\id\tens\bhL)\alpha(x)\bigr)\neq{0}$. By definition of the extended positive part of $\sN$ we then have
\[
0\neq\bhL\bigl((\omega\tens\id)\alpha(x)\bigr)<+\infty.
\]
Thus the element $R\bigl(\bigl((\omega\tens\id)\alpha(x)\bigr)^{\frac{1}{2}}\bigr)$ of $\Linf(\HH/\!\ker\alpha)$ is non-zero and square-integrable with respect to $\bh$, so by Theorem \ref{openThm}, $\hh{\HH/\!\ker\alpha}$ is an open subgroup of $\hh{\HH}$.
\end{proof}

\section{Canonical implementation of right quantum group homomorphism}\label{sectCan}

In this section we address the problem of determining the canonical implementation (\cite{impl} and Section \ref{subsAct}) of the action associated to a homomorphism of quantum groups. Let $\GG$ and $\HH$ be locally compact quantum groups and let $\Pi\colon\HH\to\GG$ be a homomorphism of quantum groups. Let  $\alpha\colon\Linf(\GG)\to\Linf(\GG)\vtens\Linf(\HH)$ be the corresponding action of $\HH$. In order to define the canonical implementation of $\alpha$ one must choose a n.s.f.~weight on $\Linf(\GG)$. We choose for this the right Haar weight $\bh$ of $\GG$.

A way to determine the canonical implementation is then provided by the proof of \cite[Proposition 4.3]{impl}. In this proposition S.~Vaes proves that in case the chosen weight is \emph{$\delta^{-1}$-invariant} (see \cite[D\'efinition 2.7]{Enock}, \cite[Definition 2.3]{impl}) the unitary implementation $U$ coincides with a unitary $V_\theta$ defined in \cite[Proposition 2.4]{impl} (cf.~\cite[Th\'eor\`eme 2.9]{Enock}). A careful examination of the proof of \cite[Proposition 4.3]{impl} shows that the claim is established on the basis of several properties of $V_\theta$ listed in \cite[Proposition 2.4]{impl}. In particular $\delta^{-1}$-invariance of the chosen weight does not play a crucial role. 

In our situation we choose an invariant weight, so \cite[Proposition 4.3]{impl} is not applicable directly. Nevertheless we already know that the action $\alpha\colon\Linf(\GG)\to\Linf(\GG)\vtens\Linf(\HH)$ is implemented by the corresponding bicharacter (see \eqref{Vimpl}). Moreover $V$ has the following properties:
\begin{itemize}
\item $\bigl(J^{\GG}\tens{J^{\hh{\HH}}}\bigr)V=V^*\bigl(J^{\GG}\tens{J^{\hh{\HH}}}\bigr)$,
\item $\bigl(\nabla^{\GG}\tens\nabla^{\hh{\HH}}\bigr)V=V\bigl(\nabla^{\GG}\tens\nabla^{\hh{\HH}}\bigr)$,
\item for any $\omega\in\Lone(\HH)$ and any $x\in\Dom(\eta)$ we have $\bigl((\id\tens\omega)V\bigr)\eta(x)=\eta\bigl((\id\tens\omega)\alpha(x)\bigr)$, where $\eta$ is the GNS map for the right Haar weight of $\GG$.
\end{itemize}
The first two properties are consequences of invariance of $V$ under $R^{\hh{\GG}}\tens{R^\HH}$ and $\tau^{\hh{\GG}}\tens\tau^\HH$ which we already used in the proof of Proposition \ref{BV2} (cf.~\cite[Proposition 2.1]{KVvN}). The last property follows from \eqref{dodatek} and Proposition \ref{Teza} from the Appendix. Indeed, using notation introduced at the beginning of Section \ref{cptSec}, for $\omega\in\Lone(\HH)$ we set $\mu=\omega\comp\Lambda_\HH\comp\pi$, where $\pi$ is the Hopf $*$-homomorphism corresponding to $\Pi$ and $\Lambda_\HH$ is the reducing morphism for $\HH$. Then by \eqref{dodatek}
\[
\mu\staru{x}=(\id\tens\mu)\bigl(\wW^\GG(x\tens\I){\wW^\GG}^*\bigr)
=(\id\tens\omega)\bigl(V(x\tens\I)V^*\bigr)
=(\id\tens\omega)\bigl(\alpha(x)\bigr)
\]
(because $\bigl(\id\tens(\Lambda_\HH\comp\pi)\bigr)\wW^\GG=V$) and
\[
\bigl((\id\tens\omega)V\bigr)\eta(x)=\bigl((\id\tens\mu)\wW^\GG\bigr)\eta(x)=\eta(\mu\staru{x})
\]
by Proposition \ref{Teza}. 

Repeating the steps of \cite[Proposition 4.3]{impl} with $\I$ instead of $\delta^{-1}$ we obtain the following:

\begin{theorem}\label{canonical}
Let $\GG$ and $\HH$ be locally compact quantum groups and let $\Pi\colon\HH\to\GG$ be a homomorphisms of quantum groups with corresponding bicharacter $V\in\Linf(\hh{\GG})\vtens\Linf(\HH)$ and right quantum group homomorphism $\alpha\colon\Linf(\GG)\to\Linf(\GG)\vtens\Linf(\HH)$. Then $V$ is the canonical implementation of $\alpha$.
\end{theorem}

Let us also remark that Theorem \ref{canonical} was actually stated and used already in \cite[Proof of Proposition 3.12]{DeCommer}.

We are now ready to use \cite[Theorem 5.3]{impl} which characterizes integrability of actions in the following terms: let $\alpha\colon\sN\to\sN\vtens\Linf(\HH)$ be an action of a locally compact quantum group $\HH$ on a von Neumann algebra $\sN$. Choose a n.s.f.~weight $\theta$ on $\sN$ and identify $\sN$ with its image in $\B(\cH_\theta)$. Furthermore let $U\in\B(\cH_\theta)\vtens\Linf(\HH)$ be the canonical implementation of $\alpha$ and let $\sN_2$ be the von Neumann algebra
\[
\sN_2=\bigl(\sN\cup{\bigl\{(\id\tens\omega)U\st}\omega\in\Lone(\HH)\bigr\}\bigr)''.
\]
Now \cite[Theorem 5.3]{impl} says that $\alpha$ is integrable if and only if there exists a normal surjective $*$-homomorphism $\rho\colon\sN\rtimes_\alpha\HH\to\sN_2$ such that
\begin{itemize}
\item $\rho\bigl(\alpha(x)\bigr)=x$ for all $x\in\sN$,
\item $\rho\bigl(\I\tens[(\id\tens\omega)\ww^\HH]\bigr)=(\id\tens\omega)U$ for all $\omega\in\Lone(\HH)$.
\end{itemize}

\section{Integrability and quantum subgroups}\label{cons}

In this section we will study homomorphisms of quantum groups with the property that the associated action is integrable. The first consequence will be that the kernel of such a homomorphism is in some sense small and the other that the quotient by the kernel can be canonically identified with the closure of the image. To make this more precise let $\Pi\colon\HH\to\GG$ be a homomorphism of quantum groups with associated bicharacter $V\in\Linf(\hh{\GG})\vtens\Linf(\HH)$. We will say that \emph{$\HH/\!\ker{\Pi}$ can be canonically identified with $\IM{\Pi}$} if there exists an isomorphism $\chi\colon\Linf\bigl(\IM{\Pi}\:\!\bigr)\to\Linf(\HH/\!\ker{\Pi})$ such that
\begin{equation}\label{chi}
(\id\tens\chi)(\ww^{\IM{\Pi}})=V.
\end{equation}
Note that indeed the right leg of $V$ belongs to $\Linf(\HH/\!\ker{\Pi})$, in fact we have
\[
V\in\Linf\bigl(\hh{\IM{\Pi}}\:\!\bigr)\vtens\Linf(\HH/\!\ker\Pi).
\]
We will write simply $\HH/\!\ker\Pi\cong\IM{\Pi}$ to denote this situation. This will not lead to confusion, as an isomorphism $\chi$ satisfying \eqref{chi} is necessarily unique.

\begin{theorem}\label{alphaInt}
Let $\GG$ and $\HH$ be locally compact quantum groups and let $\Pi\colon\HH\to\GG$ be a homomorphism of quantum groups. Assume that the corresponding right quantum group homomorphism $\alpha$ is integrable. Then
\begin{enumerate}
\item\label{alphaInt1} $\hh{\HH/\!\ker\Pi}$ is an open subgroup of $\hh{\HH}$,
\item\label{alphaInt2} $\HH/\!\ker\Pi\cong\IM{\Pi}$.
\end{enumerate}
\end{theorem}

\begin{proof}
Ad \eqref{alphaInt1}. This follows from Proposition \ref{PropNH} and Remark \ref{remPiAlpha}\eqref{remPiAlpha2}.

Ad \eqref{alphaInt2}. Let us denote $\Linf(\GG)$ by $\sN$, thus $\alpha\colon\sN\to\sN\vtens\Linf(\HH)$ and let us, as usual, denote by $V$ the bicharacter corresponding to $\Pi$. The assumption of integrability means that, by \cite[Theorem 5.3]{impl} and Theorem \ref{canonical}, we have a normal surjective $*$-homomorphism $\rho\colon\sN\rtimes_\alpha\HH\to\sN_2$, where $\sN_2$ is the von Neumann algebra generated by $\sN$ and $\bigl\{(\id\tens\omega)V\st\omega\in\Lone(\HH)\bigr\}$ such that
\begin{itemize}
\item $\rho\bigl(\alpha(x)\bigr)=x$ for all $x\in\sN$,
\item $\rho\bigl(\I\tens[(\id\tens\omega)\ww^\HH]\bigr)=(\id\tens\omega)V$ for all $\omega\in\Lone(\HH)$.
\end{itemize}
Let $\widetilde{\rho}$ denote the map $\Linf(\hh{\HH})\ni{y}\mapsto\rho(\I\tens{y})\in\sN_2$. By the second point above and the definition of $\IM\Pi$ the map $\widetilde{\rho}$ has image $\Linf\bigl(\hh{\IM\Pi}\bigr)$. Since it also preserves comultiplications, we find that $\hh{\IM\Pi}$ is an open subgroup of $\hh{\HH}$ and, as open subgroups are closed, there exists a unital injective normal $*$-homomorphism $\chi\colon\Linf\bigl(\IM\Pi\:\!\bigr)\hookrightarrow\Linf(\HH)$ such that
\begin{equation}\label{WHV}
(\widetilde{\rho}\tens\id)\ww^{\HH}=(\id\tens\chi)\ww^{\IM\Pi}
\end{equation}
(cf.~\cite[Remark 3.7]{open}). But the left hand side of \eqref{WHV} is $V$, due to properties of $\rho$, which proves the claim.
\end{proof}

\begin{theorem}\label{V2int}
Let $\GG$ and $\HH$ be locally compact quantum groups and let $\HH$ be a closed quantum subgroup of $\GG$ with right quantum group homomorphism $\alpha\colon\Linf(\GG)\to\Linf(\GG)\vtens\Linf(\HH)$ corresponding to the embedding $\HH\hookrightarrow\GG$. Then $\alpha$ is integrable.
\end{theorem}

As in the proof of Theorem \ref{alphaInt}, we will use the criterion of integrability from \cite{impl} mentioned at the end of Section \ref{sectCan}. We will in fact show that if $\HH$ is a closed subgroup of $\GG$ and we write $\sN$ for the algebra $\Linf(\GG)$ then the algebras $\sN\rtimes_\alpha\HH$ and $\sN_2$ are isomorphic with the isomorphism satisfying the requirements of \cite[Theorem 5.3]{impl}. Theorem \ref{V2int} was proved in \cite{DeCommer} by the same technique.

\begin{proof}[Proof of Theorem \ref{V2int}]
Let us begin with the remark that if $Y=(J\tens{J})\ww^{\hh{\GG}}(J\tens{J})$ then $Y\in\Linf(\GG)'\vtens\Linf(\hh{\GG})$ and $Y$ implements the opposite comultiplication on $\Linf(\hh{\GG})$:
\[
Y(u\tens\I)Y^*=\Delta_{\hh{\GG}^\op}(u),\qquad{u}\in\Linf(\hh{\GG}).
\]
Indeed, writing $\hh{R}$ for the unitary antipode of $\hh{\GG}$ we have
\[
\begin{split}
Y(u\tens\I)Y^*&=(J\tens{J})\ww^{\hh{\GG}}(J\tens{J})(u\tens\I)(J\tens{J}){\ww^{\hh{\GG}}}^*(J\tens{J})\\
&=(J\tens{J})\ww^{\hh{\GG}}\bigl(\hh{R}(u^*)\tens\I\bigr){\ww^{\hh{\GG}}}^*(J\tens{J})\\
&=(J\tens{J})\Delta_{\hh{\GG}}\bigl(\hh{R}(u^*)\bigr)(J\tens{J})\\
&=(J\tens{J})(\hh{R}\tens\hh{R})\bigl(\Delta_{\hh{\GG}^\op}(u^*)\bigr)(J\tens{J})=\Delta_{\hh{\GG}^\op}(u).
\end{split}
\]

Let us write $\sN$ for the algebra $\Linf(\GG)$. We can now define a map from $\sN\rtimes_\alpha\HH$ to the algebra $\sN_2$ which is the von Neumann algebra generated inside $\B(\Ltwo(\GG))$ by $\sN$ and $\bigl\{(\id\tens\omega)V\st\omega\in\Lone(\HH)\bigr\}$. Recall that since $\HH$ is a closed subgroup of $\GG$, we have a normal injective unital $*$-homomorphism $\gamma\colon\Linf(\hh{\HH})\hookrightarrow\Linf(\hh{\GG})$ such that
\[
(\gamma\tens\id)\ww^\HH=V.
\]

For $a\in\sN\rtimes_\alpha\HH$ we will consider the element
\begin{equation}\label{VaVY}
Y^*\bigl[(\id\tens\gamma)(V^*aV)\bigr]Y.
\end{equation}
It is not immediately clear that $V^*aV$ belongs to the domain of $\id\tens\gamma$, but it follows from the following computation. We know that elements of the form $\alpha(x)(\I\tens{y})=V(x\tens\I)V^*(\I\tens{y})$ with $x\in\sN$ and $y\in\Linf(\hh{\HH})$ span a dense subspace in $\sN\rtimes_\alpha\HH$ (\cite[Proposition 2.3]{proj}). Taking $a=V(x\tens\I)V^*(\I\tens{y})$ in \eqref{VaVY} yields
\[
\begin{split}
Y^*\bigl[(\id\tens\gamma)\bigl((x\tens\I)V^*(\I\tens{y})V\bigr)\bigr]Y
&=Y^*\bigl[\bigl((x\tens\I)(\gamma\tens\gamma)\bigl({\ww^{\HH}}^*(\I\tens{y})\ww^\HH\bigr)\bigr)\bigr]Y\\
&=Y^*\bigl[(x\tens\I)(\gamma\tens\gamma)\bigl(\Delta_{\hh{\HH}^\op}(y)\bigr)\bigr]Y\\
&=Y^*\bigl[(x\tens\I)\Delta_{\hh{\GG}^\op}\bigl(\gamma(y)\bigr)\bigr]Y\\
&=Y^*\bigl[(x\tens\I)Y\bigl(\gamma(y)\tens\I\bigr)Y^*\bigr]Y\\
&=Y^*\bigl[Y(x\tens\I)\bigl(\gamma(y)\tens\I\bigr)Y^*\bigr]Y\\
&=(x\tens\I)\bigl(\gamma(y)\tens\I\bigr)=x\gamma(y)\tens\I.
\end{split}
\]
Thus there exists a normal injective unital $*$-homomorphism $\rho\colon\sN\rtimes_\alpha\HH\to\sN_2$ such that
\[
\rho\bigl(\alpha(x)(\I\tens{y})\bigr)=x\gamma(y),\qquad{x}\in\sN,\:y\in\Linf(\hh{\HH}).
\]
In particular
\begin{itemize}
\item $\rho\bigl(\alpha(x)\bigr)=x$ for $x\in\sN$,
\item $\rho\bigl(\I\tens[(\id\tens\omega)\ww^\HH]\bigr)=(\id\tens\omega)V$ for $\omega\in\Lone(\HH)$.
\end{itemize}
Clearly $\sN$ and $\gamma\bigl(\Linf(\hh{\HH})\bigr)$ are both contained in the range of $\rho$, so $\rho$ must be surjective.
\end{proof}

\begin{theorem}\label{intThm}
Let $\GG$ and $\HH$ be locally compact quantum groups and let $\Pi\colon\HH\to\GG$ be a homomorphism of quantum groups with corresponding right quantum group homomorphism $\alpha\colon\Linf(\GG)\to\Linf(\GG)\vtens\Linf(\HH)$. Assume that
\begin{enumerate}
\item\label{intThm1} $\hh{\HH/\!\ker\Pi}$ is an open subgroup of $\hh{\HH}$,
\item $\HH/\!\ker\Pi\cong\IM{\Pi}$.
\end{enumerate}
Then $\alpha$ is integrable.
\end{theorem}

\begin{proof}
As we already noticed in the first part of the proof of Theorem \ref{alphaInt}, the image of $\alpha$ is contained in $\Linf(\GG)\vtens\Linf(\HH/\!\ker\Pi)$. Now let $\chi\colon\Linf\bigl(\IM{\Pi}\:\!\bigr)\to\Linf(\HH/\!\ker{\Pi})$ be as in \eqref{chi} and let $\widetilde{\alpha}=(\id\tens\chi^{-1})\comp\alpha$. Then $\widetilde{\alpha}$ is an action of $\IM\Pi$ on $\Linf(\GG)$ implemented by the bicharacter $(\id\tens\chi^{-1})V\in\Linf(\hh{\GG})\vtens\Linf\bigl(\IM{\Pi}\:\!\bigr)$. But we know that $\IM\Pi$ is a closed subgroup of $\GG$ via the corresponding homomorphism of quantum groups, so by Theorem \ref{V2int} the action $\widetilde{\alpha}$ is integrable. To conclude that $\alpha$ is integrable we note that due to assumption \eqref{intThm1} the Haar measure of $\HH/\!\ker\Pi$ is (proportional to) the restriction of the Haar measure of $\HH$ to $\Linf(\HH/\!\ker\Pi)$ by Corollary \ref{openCor}.
\end{proof}

\begin{remark}
Theorems \ref{alphaInt} and \ref{intThm} provide the following characterization of those homomorphisms of quantum groups $\Pi\colon\HH\to\GG$ for which the corresponding right quantum group homomorphism $\alpha\colon\Linf(\GG)\to\Linf(\GG)\vtens\Linf(\HH)$ is an integrable action: $\alpha$ is integrable if and only if the conditions \eqref{alphaInt1} and \eqref{alphaInt2} from Theorem \ref{alphaInt} hold.
\end{remark}

\begin{corollary}\label{free}
Let $\GG$ and $\HH$ be locally compact quantum groups and let $\Pi\colon\HH\to\GG$ be a homomorphism of quantum groups with corresponding right quantum group homomorphism $\alpha\colon\Linf(\GG)\to\Linf(\GG)\vtens\Linf(\HH)$. Assume that $\alpha$ is integrable and free. Then $\Pi$ identifies $\HH$ with a closed quantum subgroup of $\GG$.
\end{corollary}

\begin{proof}
The freeness of $\alpha$ means that ${\bigl\{(\omega\tens\id)\alpha(x)\st{x}}\in\Linf(\GG),\:\omega\in\Lone(\GG)\bigr\}''$ is all of $\Linf(\HH)$. In view of Theorem \ref{ThmPi}, this means that $\Linf(\HH/\!\ker{\Pi})=\Linf(\HH)$ (cf.~discussion after Definition \ref{DefKerPi}). Since $\alpha$ is assumed to be integrable, by Theorem \ref{alphaInt}\eqref{alphaInt2}, we have $\HH/\!\ker{\Pi}\cong\IM{\Pi}$ in the sense that there exists an isomorphism $\chi\colon\Linf\bigl(\IM{\Pi}\:\!\bigr)\to\Linf(\HH/\!\ker{\Pi})$ such that
\begin{equation}\label{chiWV}
(\id\tens\chi)\ww^{\IM{\Pi}}=V
\end{equation}
(cf.~remarks before Theorem \ref{alphaInt}).

Thus $\chi$ can be regarded as an isomorphism $\Linf\bigl(\IM{\Pi}\:\!\bigr)\to\Linf(\HH)$ with the property \eqref{chiWV}. Let $\hh{\chi}$ be the dual of $\chi$. Then $\hh{\chi}$ is an isomorphism $\Linf(\hh{\HH})\to\Linf\bigl(\hh{\IM{\Pi}}\:\!\bigr)$ such that $(\hh{\chi}\tens\id)\ww^{\HH}=(\id\tens\chi)\ww^{\IM{\Pi}}$, cf.~\cite{MRW}, \cite[Theorem 1.10]{DKSS}. In other words we have
\[
V=(\hh{\chi}\tens\id)\ww^\HH,
\]
where $\hh{\chi}$ is a normal, unital injective $*$-homomorphism. It follows that the bicharacter $V$ identifies $\HH$ with a closed quantum subgroup of $\GG$.
\end{proof}

We now turn our attention to quantum subgroups closed in the sense of Woronowicz (\cite[Definition 3.2]{DKSS}). Let us recall that if $\HH$ and $\GG$ are locally compact quantum groups and $\Pi\colon\HH\to\GG$ identifies $\HH$ with a \emph{closed quantum subgroup of $\GG$ in the sense of Woronowicz} if the bicharacter $V$ associated with $\Pi$ generates the \cst-algebra $\C_0(\HH)$. One of the equivalent formulations of this condition is that
\[
\bigl\{(\omega\tens\id)V\st\omega\in\B(\Ltwo(\GG))_*\bigr\}
\]
is norm dense in $\C_0(\HH)$ (\cite[Theorem 3.6]{DKSS}). In particular it is then $\sigma$-weakly dense in $\Linf(\HH)$. In what follows we will use the phrase ``Woronowicz-closed quantum subgroup'' instead of ``closed quantum subgroup in the sense of Woronowicz'' in order to make the statements more transparent.

\begin{corollary}\label{W2V}
Let $\GG$ and $\HH$ be locally compact quantum groups and let $\Pi\colon\HH\to\GG$ be a homomorphism of quantum groups which identifies $\HH$ with a Woronowicz-closed quantum subgroup of $\GG$. If the corresponding right quantum group homomorphism $\alpha\colon\Linf(\GG)\to\Linf(\GG)\vtens\Linf(\HH)$ is integrable then $\Pi$ identifies $\HH$ with a closed quantum subgroup of $\GG$.
\end{corollary}

\begin{proof}
It is clear from the remark preceding the corollary that if $\HH$ is a Woronowicz-closed quantum subgroup of $\GG$ then $\alpha$ is free. As it is assumed to be integrable, the fact that it is a closed quantum subgroup as defined in Section \ref{qsub} follows from Corollary \ref{free}.
\end{proof}

\begin{remark}
\noindent
\begin{enumerate}
\item Corollary \ref{W2V} together with Theorem \ref{V2int} show that closed quantum subgroups are precisely Woronowicz-closed quantum subgroups with integrable corresponding right quantum group homomorphism.
\item This gives as an easy consequence an immediate proof of \cite[Theorem 6.1]{DKSS} as any action of a compact quantum group is integrable.
\item Finally, Corollary \ref{W2V} is really a characterization of those homomorphisms of quantum groups $\Pi\colon\HH\to\GG$ which identify $\HH$ with a closed quantum subgroup of $\GG$. Indeed, if $\HH$ is a closed quantum subgroup then the action corresponding to $\Pi$ is free (because a closed quantum subgroup of $\GG$ is a Woronowicz-closed quantum subgroup of $\GG$, \cite[Theorem 3.5]{DKSS}) and it is integrable by Theorem \ref{V2int}.
\end{enumerate}
\end{remark}

\section{Appendix}\label{appendix}

\subsection{Classical case}

Let us discuss the results we obtained for locally compact quantum groups in the special case when the groups involved are in fact classical. In what follows the symbol $\vN(K)$ will denote the (right) group von Neumann algebra of a locally compact group $K$. 

Thus let $G$ and $H$ be locally compact groups and let $\Pi\colon{H}\to{G}$ be a continuous homomorphism. The three equivalent descriptions of $\Pi$ used when dealing with quantum groups take the following form:
\begin{itemize}
\item the bicharacter $V\in\Linf(\hh{G})\vtens\Linf(H)=\vN(G)\vtens\Linf(H)\cong\Linf(H,\vN(G))$ is the representation of $H$ on $\Ltwo(G)$ by right shifts along the image of $\Pi$:
\[
(V_h\psi)(g)=\psi\bigl(g\Pi(h)\bigr),\qquad\psi\in\Ltwo(G),\:g\in{G},\:h\in{H}.
\]
\item The action $\alpha\colon\Linf(G)\to\Linf(G)\vtens\Linf(H)\cong\Linf(G\times{H})$ is the pull-back of the natural right action of $H$ on $\Linf(G)$ arising from $\Pi$:
\[
g\cdot{h}=g\Pi(h),\qquad{g}\in{G},\:h\in{H},
\]
so that for $f\in\Linf(G)$ the function $\alpha(f)\in\Linf(G\times{H})$ is $\alpha(f)(g,h)=f\bigl(g\Pi(h)\bigr)$.
\item The Hopf $*$-homomorphism $\pi\in\Mor(\C_0(G),\C_0(H))$ is the pre-composition with $\Pi$:
\[
\pi(f)=f\comp\Pi,\qquad{f}\in\C_0(G).
\]
\end{itemize}

Let us fix $G$, $H$ and $\Pi$ (so we also have $V$, $\alpha$ and $\pi$). We leave to the reader the verification that
\[
{\bigl\{(\id\tens\zeta)V\st\zeta\in\Lone(H)\bigr\}}^{\sigma-\text{\rm{weak closure}}\hspace{-\sw}}
\]
and
\[
{\bigl\{(\phi\tens\id)V\st\phi}\in\vN(G)_*\bigr\}^{\sigma-\text{\rm{weak closure}}\hspace{-\sw}}
\]
coincide respectively with $\vN\bigl(\IM{\Pi}\!\:\bigr)$ canonically embedded in $\vN(G)$ and $\Linf(H/\!\ker{\Pi})$ viewed as a subalgebra of $\Linf(H)$ (functions constant on cosets of $\ker{\Pi}$). 

It is easy to see that the action $\alpha$ is integrable if and only if the set of positive $f\in\Linf(G)$ such that the integral
\[
\int\limits_Hf\bigl(\,\cdot\:\Pi(h)\bigr)d\bhL(h)
\]
exists in $\Linf(G)$ is $\sigma$-weakly dense in $\Linf(G)_+$.

Rewriting the results of Section \ref{cons} for classical groups we find the following:

\begin{theorem}\label{classInt}
Let $G$ and $H$ be locally compact groups and let $\Pi\colon{H}\to{G}$ be a continuous homomorphism. Let $\alpha$ be the action of $H$ on $\Linf(G)$ corresponding to $\Pi$. Then $\alpha$ is integrable if and only if
\begin{enumerate}
\item\label{classInt1} $\ker{\Pi}$ is compact,
\item the image of $\Pi$ is closed and $H/\!\ker{\Pi}$ is homeomorphic to $\mathrm{im}\,\Pi$.
\end{enumerate}
\end{theorem}

We leave the details to the reader, but let us only mention that condition \eqref{classInt1} of Theorem \ref{classInt} is equivalent to $\hh{H/\!\ker{\Pi}}$ being an open quantum subgroup of $\hh{H}$ by \cite[Theorem 7.2]{open}. In the special case of an injective morphism we obtain the following corollary:

\begin{corollary}
Let $G$ and $H$ be locally compact groups and let $\Pi\colon{H}\to{G}$ be a continuous injective homomorphism. Then the image of $\Pi$ is closed and $\Pi$ is a homeomorphism onto $\mathrm{im}\,\Pi$ if and only if the associated action of $H$ on $\Linf(G)$ is integrable.
\end{corollary}

\subsection{Characterization of compact quantum subgroups}\label{cptSec}

In this section we will present a characterization of compact quantum subgroups somewhat analogous to that of open quantum subgroups given in Theorem \ref{openThm} (cf.~also Corollary \ref{final}). To that end let us recall from \cite[Definition 4.1]{impl} (see also \cite[Lemma 3.11]{eqhs}) that if $\HH$ and $\GG$ are locally compact quantum groups and $\HH$ is identified with a closed subgroup of $\GG$ via $\gamma\colon\Linf(\hh{\HH})\hookrightarrow\Linf(\hh{\GG})$ then the \emph{measured homogeneous space} $\GG/\HH$ is described by the left coideal $\Linf(\GG/\HH)\subset\Linf(\GG)$ defined as the co-dual of $\gamma\bigl(\Linf(\hh{\HH})\bigr)\subset\Linf(\hh{\GG})$ (see Section \ref{lcqgs}).

\begin{theorem}\label{cptThm}
Let $\GG$ and $\HH$ be locally compact quantum groups and let $\HH$ be a closed quantum subgroup of $\GG$. Then $\HH$ is compact if and only if there exists a non-zero element $x\in\Linf(\GG/\HH)$ which is square integrable with respect to the Haar measure of $\GG$.
\end{theorem}

For the proof of Theorem \ref{cptThm} we need a few results about convolutions. Let $\GG$ be a locally compact quantum group. Then we have Banach algebra structures on spaces of functionals like $\Lone(\GG)$ or $\C_0^\uu(\GG)^*$. For the purposes of this section we will denote the (convolution) products on these spaces respectively by $\starr$ and $\staru$:
\[
\begin{aligned}
\omega_1\starr\omega_2&=(\omega_1\tens\omega_2)\comp\Delta_\GG,&&\quad\omega_1,\omega_2\in\Lone(\GG),\\
\mu_1\staru\mu_2&=(\mu_1\tens\mu_2)\comp\Delta_\GG^\uu,&&\quad\mu_1,\mu_2\in\C_0^\uu(\GG)^*.
\end{aligned}
\]
Moreover the adjoint of the reducing morphism $\Lambda_\GG\colon\C_0^\uu(\GG)\twoheadrightarrow\C_0(\GG)\subset\Linf(\GG)$ provides an isometry of $\iota$ of $\Lone(\GG)$ into $\C_0^\uu(\GG)^*$. The image of $\iota$ is a closed two sided ideal in $\C_0^\uu(\GG)^*$ (\cite[Proposition 5.3]{univ}) and denoting by $\cR_\mu$ the operator
\[
\omega\longmapsto\iota^{-1}\bigl(\iota(\omega)\staru\mu\bigr)
\]
we obtain a bounded map $\Lone(\GG)\to\Lone(\GG)$. The adjoint of this map, denoted by $\Linf(\GG)\ni{x}\mapsto\mu\staru{x}\in\Linf(\GG)$, is $\sigma$-weakly continuous and is in fact given by
\begin{equation}\label{dodatek}
\mu\staru{x}=(\id\tens\mu)\bigl(\wW^\GG(x\tens\I){\wW^\GG}^*\bigr)
\end{equation}
(this follows e.g.~from \cite[Proof of Proposition 5.3]{univ}). In particular for any $\omega\in\Lone(\GG)$ we have
\[
\iota(\omega)\staru{x}=\omega\starr{x},
\]
where $\omega\starr{x}$ is by definition $(\id\tens\omega)\Delta_\GG(x)$. Also it is not difficult to check that
\begin{equation}\label{omx}
\omega\starr(\mu\staru{x})=\cR_\mu(\omega)\starr{x}
\end{equation}
for all $\omega\in\Lone(\GG)$, $\mu\in\C_0^\uu(\GG)^*$ and $x\in\Linf(\GG)$.

By \cite[Lemma 3.4]{knr}, for $\mu\in\C_0^\uu(\GG)^*$ and $x\in\Linf(\GG)_+$ we have
\[
\bh(\mu\staru{x})=\mu(\I)\bh(x).
\]
In particular, it follows from the Kadison inequality that if $x\in\Linf(\GG)$ is square integrable for $\bh$ and $\mu\in\C_0^\uu(\GG)^*_+$ then $\mu\staru{x}$ is also square integrable for $\bh$.

\begin{proposition}\label{Teza}
Let $\GG$ be a locally compact quantum group. Then for $\mu\in\C_0^\uu(\GG)^*$ and $x\in\Dom(\eta)$ 
\begin{equation}\label{teza}
\eta(\mu\staru{x})=\bigl((\id\tens\mu)(\wW^\GG)\bigr)\eta(x)
\end{equation}
\end{proposition}

\begin{proof}
It is enough to prove this for positive $\mu$ because both sides of \eqref{teza} are linear in this variable.

Take $\omega\in\Lone(\GG)$. Using \cite[Proposition 1.10, formula (1.7)]{mnw} and \cite[Theorem 3.5, formula (3.6)]{mnw} we find that on one hand
\begin{equation}\label{onehand}
\eta\bigl((\cR_\mu\omega)\starr{x}\bigr)=(\id\tens{\cR_\mu\omega})(\ww^\GG)\eta(x)
\end{equation}
and on the other hand, by \eqref{omx},
\begin{equation}\label{otherhand}
\eta\bigl((\cR_\mu\omega)\starr{x})\bigr)=\eta\bigl(\omega\starr(\mu\staru{x})\bigr)=(\id\tens\omega)(\ww^\GG)\eta(\mu\staru{x}).
\end{equation}
In order to analyze the operator $(\id\tens{\cR_\mu\omega})(\ww^\GG)$ we compute:
\begin{equation}\label{onemorehand}\begin{split}
(\id\tens{\cR_\mu\omega})(\ww^\GG)&=\bigl(\id\tens\iota^{-1}(\iota(\omega)\staru\mu)\bigr)\ww^\GG\\
&=\bigl(\id\tens\iota(\omega)\staru\mu\bigr)\wW^\GG\\
&=\bigl(\id\tens\iota(\omega)\tens\mu\bigr)(\id\tens\Delta_\GG^\uu)\wW^\GG\\
&=\bigl(\id\tens\iota(\omega)\tens\mu\bigr)\wW^\GG_{12}\wW^\GG_{13}\\
&=\bigl(\id\tens\omega\tens\mu\bigr)\ww^\GG_{12}\wW^\GG_{13}\\
&=(\id\tens\omega)(\ww^\GG)(\id\tens\mu)(\wW^\GG).
\end{split}
\end{equation}

Comparing \eqref{onehand} and \eqref{otherhand} we find using \eqref{onemorehand} that
\[
(\id\tens\omega)(\ww^\GG)\eta(\mu\staru{x})=(\id\tens\omega)(\ww^\GG)(\id\tens\mu)(\wW^\GG)\eta(x).
\]
As operators of the form $(\id\tens\omega)(\ww^\GG)$ can strongly approximate the identity operator, we have proved \eqref{teza}.
\end{proof}

\begin{proof}[Proof of Theorem \ref{cptThm}]
Assume first that $\HH$ is a closed quantum subgroup of $\GG$ with a non-zero $x\in\Linf(\GG/\HH)$ square integrable for $\bh$.

Thus we have a normal unital and injective $*$-homomorphism $\gamma\colon\Linf(\hh{\HH})\hookrightarrow\Linf(\hh{\GG})$, the corresponding bicharacter $V=(\gamma\tens\id)\ww^\HH$ and the surjective Hopf $*$-homomorphism $\pi\colon\C_0^\uu(\GG)\to\C_0^\uu(\HH)$ such that
\[
V=(\id\tens\Lambda_\HH)(\id\tens\pi)(\wW^\GG).
\]
Moreover $\Linf(\hh{\HH})$ (now identified with a subalgebra of $\Linf(\hh{\GG})$) is the $\sigma$-weak closure of
\[
\big\{(\id\tens\omega)V\st\omega\in\B(\Ltwo(\HH))_*\bigr\}=\big\{(\id\tens\omega)V\st\omega\in\Lone(\HH)\bigr\}.
\]
For the purposes of this proof we will identify $\Linf(\hh{\HH})$ with its image in $\Linf(\hh{\GG})$ under $\gamma$. 

Let $P\in\B(\Ltwo(\GG))$ be the projection onto the one dimensional subspace spanned by $\eta(x)$. We will first show that for any $\omega\in\Lone(\HH)$ and $y=(\id\tens\omega)V$ we have
\begin{equation}\label{yP}
yP=\omega(\I)P.
\end{equation}
Indeed, let $\mu=\omega\comp\Lambda_\HH\comp\pi$. Then $y=(\id\tens\mu)\wW^\GG$, so
\begin{equation}\label{yeta}
y\eta(x)=(\id\tens\mu)\wW^\GG\eta(x)=\eta(\mu\staru{x})
\end{equation}
by Proposition \ref{Teza}.

Now $x\in\Linf(\GG/\HH)=\Linf(\hh{\HH})'\cap\Linf(\GG)$, so
\[
\begin{split}
\mu\staru{x}&=(\id\tens\mu)\bigl(\wW^\GG(x\tens\I){\wW^\GG}^*\bigr)\\
&=(\id\tens\omega)\bigl(V(x\tens\I)V^*\bigr)\\
&=(\id\tens\omega)(x\tens\I)=\omega(\I){x},
\end{split}
\]
which in view of \eqref{yeta} proves \eqref{yP}.

Equality \eqref{yP} for all $\omega$ shows that for any $y\in\Linf(\hh{\HH})$ the operator $yP$ is proportional to $P$. Thus we can define a $\sigma$-weakly continuous functional $\hh{\eps}\colon\Linf(\hh{\HH})\to\CC$ by
\[
yP=\hh{\eps}(y)P
\]
(it is not difficult to check that $\hh{\eps}$ is in fact a $*$-homomorphism and, in particular, a state).

For $y=(\id\tens\omega)V$ we get $\hh{\eps}(y)=\omega(\I)$ or, in other words
\[
(\hh{\eps}\tens\omega)V=\omega(\I).
\]
As this holds for all $\omega$, we find that
\[
(\hh{\eps}\tens\id)V=\I.
\]
It follows that $\I$ belongs to $\C_0(\HH)$, so $\HH$ is compact.

Assume now that $\HH$ is a compact subgroup of $\GG$. In \cite[Theorem 5.1]{exa} a conditional expectation $\C_0^\uu(\GG)\to\C_0^\uu(\GG/\HH)$ was defined (see \cite[Example 5.3(2)]{eqhs}), but in fact the arguments used there show that we also have the conditional expectation on the reduced level: $E\colon\C_0(\GG)\to\C_0(\GG/\HH)$
\[
E(a)=(\id\tens\bh_{\HH})\bigl(V(a\tens\I)V^*\bigr)
\]
where $\bh_{\HH}$ is the Haar measure of $\HH$. Note that $E(a)=\mu\staru{a}$, where $\mu=\bh_{\HH}\comp\Lambda_\HH\comp\pi$ (cf.~the first part of the proof). By remarks preceding the proof of the theorem we find that there are many non-zero elements of $\Linf(\GG/\HH)$ which are square integrable for the right Haar measure of $\GG$.
\end{proof}

One can remark that the character $\hh{\eps}$ constructed in the proof of Theorem \ref{cptThm} can easily be shown to satisfy $(\hh{\eps}\tens\id)\comp\Delta_{\hh{\HH}}=\id$, so that $\hh{\eps}$ is the counit on $\Linf(\hh{\HH})$.

We end with a re-statement of Theorem \ref{cptThm} in terms of integrability (cf.~Corollary \ref{openInt}).

\begin{corollary}\label{final}
Let $\GG$ and $\HH$ be locally compact quantum groups and let $\HH$ be a closed quantum subgroup of $\GG$. Then $\HH$ is compact if and only if the left action of $\GG$ on $\GG/\HH$ given by $\Delta_\GG$ is integrable.
\end{corollary}

\subsection*{Acknowledgements}

The second author was partially supported by the Ministry of Science of Iran and wishes to thank the Department of Mathematical Methods in Physics, Faculty of Physics, University of Warsaw for warm hospitality. The first and third authors were partially supported by the National Science Center (NCN) grant no.~2015/17/B/ST1/00085.

\end{document}